\documentclass[11pt,reqno,final]{amsart}

\usepackage{amsmath}
\usepackage{amsthm}
\usepackage{amsfonts}  
\usepackage{amssymb}   
\usepackage{latexsym}
\usepackage[mathscr]{eucal}
\usepackage{amscd}

\theoremstyle{plain}
\newtheorem{theorem}{Theorem}

\newtheorem{proposition}[theorem]{Proposition}
\newtheorem{corollary}[theorem]{Corollary}
\newtheorem{lemma}[theorem]{Lemma}

\theoremstyle{remark}
\newtheorem{remark}{Remark}
\newtheorem{example}{Example}

\DeclareMathAlphabet\mathoo{U}{eur}{b}{n}
\DeclareMathOperator{\Real}{Re}

\usepackage[bookmarksnumbered, colorlinks, plainpages]{hyperref}
\hypersetup{colorlinks=true,linkcolor=red, anchorcolor=green, citecolor=cyan, urlcolor=red, filecolor=magenta, pdftoolbar=true}

\begin{document}\large
\title[Green's function in the case of a triangular coefficient]{Green's function of the problem of bounded solutions\\ in the case of a block triangular coefficient}

\author{V.G. Kurbatov}
 \email{kv51@inbox.ru}
 \address{Department of Mathematical Physics,
Voronezh State University\\ 1, Universitetskaya Square, Voronezh 394018, Russia}

\author{I.V. Kurbatova}
 \email{la\_soleil@bk.ru}
 \address{Department of Software Development and Information Systems Administration,
Vo\-ro\-nezh State University\\ 1, Universitetskaya Square, Voronezh 394018, Russia}

\subjclass{47A60; 47A80; 34B27; 34B40; 34D09}

\keywords{bounded solutions problem; Green's function; divided difference with operator arguments; block matrix; causal operator}

\begin{abstract} It is known that the equation $x'(t)=Ax(t)+f(t)$, where $A$ is a bounded linear operator, has a unique bounded solution $x$ for any bounded continuous free term~$f$ if and only if the spectrum of the coefficient $A$ does not intersect the imaginary axis. The solution can be represented in the form
\begin{equation*}
x(t)=\int_{-\infty}^{\infty}\mathcal G(s)f(t-s)\,ds.
\end{equation*}
The kernel $\mathcal G$ is called Green's function. In this paper, the case when $A$ admits a representation by a block triangular operator matrix is considered. It is shown that the blocks of $\mathcal G$ are sums of special convolutions of Green's functions of diagonal blocks of $A$.
\end{abstract}

\maketitle

\section*{Introduction}\label{s:Introduction}
We consider the equation
\begin{equation}\label{e:eq}
x'(t)-Ax(t)=f(t),
\end{equation}
where $A$ is a linear bounded operator acting in a Banach space $X$. We assume that $f$ is continuous.

The \emph{bounded solutions problem} is the problem of finding a bounded solution~$x$ that corresponds to a bounded free term $f$. The bounded solutions problem is closely connected with the problem of the exponential dichotomy of solutions. For the discussion of the bounded solutions problem from different points of view and related questions, see~\cite{Akhmerov-Kurbatov,BaskakovMS15:eng,Boichuk-Pokutnii,Chicone-Latushkin99,
Daletskii-Krein:eng,
Godunov94:ODE:eng,Hartman02:eng,Henry81:eng,
Kurbatov-Kurbatova17:CMAM,Massera-Schaffer:eng,
Pankov1990:eng,Pechkurov12:eng
} and the references therein.

It is well known (see Theorem~\ref{t:Green}) that equation~\eqref{e:eq} has a unique bounded solution $x$ for any bounded continuous free term~$f$ if and only if the spectrum of the coefficient $A$ is disjoint from the imaginary axis. In this case, the solution can be represented in the form
\begin{equation*}
x(t)=\int_{-\infty}^{\infty}\mathcal G(s)f(t-s)\,ds.
\end{equation*}
The kernel $\mathcal G$ is called \emph{Green's function}.

In this paper, we consider the case when $A$ admits a representation in the form of a block triangular matrix~\eqref{e:triangular mat}.
The simplest $2\times2$ matrix representation of the coefficient $A$ is naturally induced by the decomposition of the space $X$ into the direct sum $X_-\oplus X_+$ of two spectral subspaces related to the parts of $\sigma(A)$ that lie in the left and right complex half-planes. This matrix representation is diagonal, but it can be `bad' in the sense that the corresponding projectors have large norms; in such a case it may be convenient to replace one of the subspaces by the orthogonal (or close to orthogonal) complement of the other; as a result one will arrive at a triangular matrix representation of $A$. Similarly, the spectrum of $A$ may be divided into clusters; so, it is again natural to use a diagonal or triangular matrix representation; the phenomenon of clusterization is discussed, e.g., in~\cite[lecture 12]{Godunov02:eng},~\cite{Davies-Higham03,Kurbatov-Kurbatova-LMA16}. Representation by triangular operator matrices is also natural for causal operators; in their turn, causal operators are widely used in control theory~\cite{Desoer-Vidyasagar,Feintuch-Saeks,Willems} and functional differential equations~\cite{Kurbatov-SMZh75:eng,Kurbatov-81DE:eng,Kurbatov99}. For other aspects of the theory of triangular operator matrices, see~\cite{Ceballos-Nunez-Tenorio,Cheung,Ghahramani,Gohberg-Goldberg-Kaashoek-II,Gohberg-Krein:eng,
Higham87:SIAM_Rev,Kadison-Singer,
Kressner-Luce-Statti17:ArXiv,Parlett76} and the references therein.

The main results of this paper are Theorems~\ref{t:exp of block tri mat} and~\ref{t:g_t[n]}; see also Theorem~\ref{t:Green}. From these theorems, it follows that Green's function is also induced by a triangular matrix and its blocks can be represented as the sums of special convolutions of Green's functions of the diagonal blocks of $A$; see Example~\ref{ex:g_t of block triang matr}.

Similar representations and related formulas for the fundamental solution of equation~\eqref{e:eq} were proposed, discussed, and applied by many authors~\cite{Carbonell-Jimenez-Pedroso08,Davis_Ch73,Dieci-Papini,Feynman51,Goodwin-Kuprov15,
Kenney-Laub98,Lutzky68,Najfeld-Havel95,
Rosenbloom1967,Van_Loan77,Van_Loan78,Wilcox68};
such formulas are widely used in numerical methods and other applications. We repeat some of these results in this paper (i) for the convenience of their comparison with our results connected with Green's function, and because (ii) we propose a new proof for them, (iii) and discuss the infinite-dimensional case, which requires some additional considerations in the proof; see Section~\ref{s:c-spectrum}.

The paper is organized as follows. In Section~\ref{s:operator functions}, we recall the definition of an analytic function with an operator argument. In Section~\ref{s:DE}, we describe the representation of the fundamental solution of initial value problem and Green's function of bounded solutions problem in the form of the analytic functions $\exp_{\pm,t}$ and $g_t$, respectively, of the coefficient~$A$. In Section~\ref{s:c-spectrum}, we discuss the subalgebra of operators induced by block triangular matrices. This subalgebra is not full, which leads to some technical difficulties in the subsequent presentation. In Section~\ref{s:func of block triang matr}, we describe (Theorem~\ref{t:func block triang matr}) a representation of blocks of an analytic function $f$ of a triangular matrix via contour integrals. In Section~\ref{s:divided differences}, the main terms of the formula from Theorem~\ref{t:func block triang matr} are represented as divided differences of $f$ with operator arguments (Theorem~\ref{t:boxdot=boxtimes}). In Section~\ref{s:exp[n] and g[n]}, we show that divided differences of $\exp_{\pm,t}$ and $g_t$ can be represented as convolutions with respect to the variable $t$ of functions of one variable (Theorem~\ref{t:exp_+[n] and g_t[n]}). In Section~\ref{s:exp and g of A}, we describe a representation of divided differences of $\exp_{\pm,t}$ and $g_t$ with operator arguments (Theorem~\ref{t:g_t[n]}). The combination of Theorems~\ref{t:exp of block tri mat} and~\ref{t:g_t[n]} allows one to represent the blocks of the fundamental solution of initial value problem and Green's function of the bounded solutions problem as special convolutions of the functions $\exp_{\pm,t}$ and $g_t$ applied to the diagonal blocks of $A$ (Examples~\ref{ex:exp of block triang matr} and~\ref{ex:g_t of block triang matr}).

\section{Functions of operators}\label{s:operator functions}
Let $X$ and $Y$ be non-zero complex Banach spaces. We denote by $\mathoo B(X,Y)$ the set of all bounded linear operators $A:\; X\to Y$. If $X=Y$, we use the brief notation $\mathoo B(X)$. The symbol $\mathbf1=\mathbf1_X$ stands for the identity operator from $\mathoo B(X)$.

Let $\mathoo B$ be a non-zero complex Banach algebra~\cite{Bourbaki_Theories_Spectrales:eng,Hille-Phillips:eng,Rudin-FA:eng} with the unit $\mathbf1$ (unital algebra). The main example of a unital Banach algebra is the algebra $\mathoo B(X)$; another important example is the algebra of all $n\times n$ matrices, $n\in\mathbb N$.

A subset $\mathoo R$ of an algebra $\mathoo B$ is called a \emph{subalgebra} if $A+B,\lambda A,AB\in\mathoo R$ for all $A,B\in\mathoo R$ and $\lambda\in\mathbb C$.
If the unit $\mathbf1$ of an algebra $\mathoo B$ belongs to its subalgebra $\mathoo R$, then $\mathoo R$ is called a \emph{subalgebra with a unit} or a \emph{unital subalgebra}.

A unital subalgebra $\mathoo R$ of a unital algebra $\mathoo B$ is called~\cite[ch.~1,~\S~3.6]{Bourbaki_Theories_Spectrales:eng} {\it full} if
it possesses the property:
if for $B\in\mathoo R$ there exists
$B^{-1}\in\mathoo B$ such that $BB^{-1}=B^{-1}B=\mathbf 1$, then $B^{-1}\in\mathoo R$.
Below (Remark~\ref{r:B_c is not full}) we will see that the subalgebra of all block triangular matrices is not always full.

Let $\mathoo B$ be a (nonzero) unital algebra and $A\in\mathoo B$. The set of all $\lambda\in\mathbb C$
such that the element $\lambda\mathbf1-A$ is not invertible is called the {\it spectrum}
of the element $A$ (in the algebra $\mathoo B$) and is denoted by the symbol $\sigma(A)$ or $\sigma_{\mathoo B}(A)$. The complement
$\rho(A)=\rho_{\mathoo B}(A)=\mathbb C\setminus\sigma(A)$ is called the \emph{resolvent set\/} of~$A$. The function
$
R_\lambda=(\lambda\mathbf1-A)^{-1}
$
is called the \emph{resolvent\/} of the element~$A$. The \emph{spectral radius} $r(A)=r_{\mathoo B}(A)$ is the radius of the smallest closed circle in $\mathbb C$ with center at 0 that contains $\sigma(A)$.

\begin{proposition}[{\rm\cite[Chap. 1, Sec. 4, Theorem 3]{Bourbaki_Theories_Spectrales:eng},~\cite[Theorem 10.18]{Rudin-FA:eng}}]\label{p:spectrum in a subalgebra}
Let $\mathoo R$ be a closed unital subalgebra of a unital algebra $\mathoo B$. Then the spectrum $\sigma_{\mathoo R}(A)$ of an element $A\in\mathoo R$ in the algebra $\mathoo R$ is the union of the spectrum $\sigma(A)$ of $A$ in the algebra $\mathoo B$ and {\rm(}possibly empty{\rm)} collection of bounded connected components of the resolvent set $\rho(A)$.
In particular, the spectral radii $r_{\mathoo R}(A)$ and $r_{\mathoo B}(A)$ coincide.
\end{proposition}

Let $A\in\mathoo B$ and let $U\subseteq\mathbb C$ be an open set that contains the spectrum $\sigma(A)$. The set $U$ must not be connected.
Let $f:\,U\to\mathbb C$ be an analytic function. The \emph{function $f$ of the element} $A$ is defined~\cite[ch.~V, \S~1]{Hille-Phillips:eng},~\cite[p.~17]{Daletskii-Krein:eng} by the formula
\begin{equation}\label{e:def of f(A)}
f(A)=\frac1{2\pi i}\int_\Gamma f(\lambda)(\lambda\mathbf1-A)^{-1}\,d\lambda,
\end{equation}
where the contour $\Gamma$ surrounds the set $\sigma_{\mathoo B}(A)$ in the counterclockwise direction and the function $f$ is analytic inside $\Gamma$.

\begin{proposition}[{\rm\cite[Theorem 5.2.5]{Hille-Phillips:eng},~\cite[Theorem 10.27]{Rudin-FA:eng}}]\label{p:func calc}
The mapping $f\mapsto f(A)$ preserves algebraic operations, i.~e.,
\begin{align*}
(f+g)(A)&=f(A)+g(A),\\
(\alpha f)(A)&=\alpha f(A),\\
(fg)(A)&=f(A)g(A),
\end{align*}
where $f+g$, $\alpha f$, and $fg$ are defined pointwise.
\end{proposition}

\begin{corollary}\label{c:r_0}
For the function $r_{\lambda_0}(\lambda)=\frac1{\lambda_0-\lambda}$, $\lambda_0\in\rho(A)$, we have
\begin{equation*}
r_{\lambda_0}(A)=(\lambda_0\mathbf1-A)^{-1}.
\end{equation*}
\end{corollary}
\begin{proof}
The proof follows from Proposition~\ref{p:func calc}.
\end{proof}

\section{The differential equation with a constant coefficient}\label{s:DE}
In this Section, we describe three analytic functions that are closely related to the representation of solutions of linear differential equations with constant coefficients.

For $\lambda\in\mathbb C$ and $t\in\mathbb R$, we consider the functions
\begin{align*}
\exp_{+,\,t}(\lambda)&=\begin{cases}
e^{\lambda t}, & \text{if $t>0$},\\
0, & \text{if $t<0$},\end{cases}\\
\exp_{-,\,t}(\lambda)&=\begin{cases}
0, & \text{if $t>0$},\\
-e^{\lambda t}, & \text{if $t<0$},\end{cases}
\\
g_t(\lambda)&=\begin{cases}
\exp_{-,\,t}(\lambda), & \text{if $\Real\lambda>0$},\\
\exp_{+,\,t}(\lambda), & \text{if $\Real\lambda<0$}.
\end{cases}
\end{align*}
These functions are undefined for $t=0$. The function $g_t$ is also undefined for $\Real\lambda=0$. For any fixed $t\neq0$, all three functions are analytic on their domains.

Let $X$ be a Banach space and $A\in\mathoo B(X)$. We consider the differential equation
\begin{equation}\label{e:x'=Ax+f}
x'(t)=Ax(t)+f(t),\qquad t\in\mathbb R.
\end{equation}

We recall two well-known theorems. The first theorem shows that $\exp_{\pm,\,(\cdot)}(A)$ are fundamental solutions of the initial value problems.
\begin{theorem}[{\rm\cite[ch.~1, \S~4]{Daletskii-Krein:eng}, \cite[ch.~IV, Corollary 2.1]{Hartman02:eng}}]\label{t:IVP}
Let $f:\,\mathbb R\to X$ be a continuous function.
The solution of the initial value problem
\begin{align*}
x'(t)&=Ax(t)+f(t),\qquad t>0,\\
x(0)&=0
\end{align*}
is the function
\begin{equation*}
x(t)=\int_0^t\exp_{+,\,s}(A)\,f(t-s)\,ds,\qquad t>0.
\end{equation*}
The solution of the initial value problem
\begin{align*}
x'(t)&=Ax(t)+f(t),\qquad t<0,\\
x(0)&=0
\end{align*}
is the function
\begin{equation*}
x(t)=\int_t^0\exp_{-,\,s}(A)\,f(t-s)\,ds,\qquad t<0.
\end{equation*}
\end{theorem}

The function $t\mapsto\exp_{+,\,t}(A)$ is called~\cite{Hartman02:eng} the \emph{fundamental solution} for equation~\eqref{e:x'=Ax+f}.

Now we turn to the \emph{bounded solutions problem}, i.e. the problem of seeking bounded solution $x:\,\mathbb R\to X$ under the assumption that the free term $f:\,\mathbb R\to X$ is a bounded function.

\begin{theorem}[{\rm\cite[Theorem 4.1, p.~81]{Daletskii-Krein:eng}}]\label{t:Green}
Let $A\in\mathoo B(X)$.
Equation~\eqref{e:x'=Ax+f} has a unique bounded on $\mathbb R$ solution $x$ for any bounded continuous function $f$ if and only if the spectrum $\sigma(A)$ of $A$ does not intersect the imaginary axis.
This solution admits the representation
\begin{equation*}
x(t)=\int_{-\infty}^\infty \mathcal G(s)f(t-s)\,ds,
\end{equation*}
where
\begin{equation*}
\mathcal G(t)=g_t(A),\qquad t\neq0.
\end{equation*}
\end{theorem}
The function $\mathcal G$ is called~\cite{Daletskii-Krein:eng} the \emph{Green's function} of the bounded solutions problem for equation~\eqref{e:x'=Ax+f}.

\section{Causal spectrum of a block triangular matrix}\label{s:c-spectrum}
Let a Banach space $X$ be represented as the direct sum of its closed nonzero subspaces $X_i$, $i=1,\dots,n$:
\begin{equation*}
X=X_1\oplus X_2\oplus\dots\oplus X_n.
\end{equation*}
This means that every $x\in X$ can be uniquely represented in the form
\begin{equation*}
x=x_1+x_2+\dots+x_n,
\end{equation*}
where $x_i\in X_i$, $i=1,\dots,n$. It is easy to prove that the norm on $X$ is equivalent to the norm
\begin{equation*}
\lVert x\rVert=\lVert x_1\rVert+\lVert x_2\rVert+\dots+\lVert x_n\rVert.
\end{equation*}
We denote by $\mathoo M=\mathoo M(X_1,X_2,\dots,X_n)$ the set of all operator matrices
\begin{equation*}
\{\,T_{ij}\in\mathoo B(X_j,X_i):\,i,j=1,\dots,n\,\}.
\end{equation*}
We endow $\mathoo M$ with the norm $\lVert \{\,T_{ij}\,\}\rVert=\max_j\sum_{i=1}^n\lVert T_{ij}\rVert$. It is easy to show that $\mathoo M$ is a unital Banach algebra with respect to the usual matrix multiplication, and the Banach algebra $\mathoo M$ is isomorphic (not isometrically) to the algebra $\mathoo B(X)$. As usual, we do not distinguish very carefully matrices and operators induced by them.

We denote by $\mathoo M^+=\mathoo M^+(X_1,X_2,\dots,X_n)$ the set of all lower triangular matrices
\begin{equation}\label{e:triangular mat}
\begin{pmatrix}
   A_{11} & 0 & \dots & 0 & 0 \\
   A_{21} & A_{22} & \dots & 0 & 0 \\
   \dots & \dots & \dots & \dots & \dots \\
   A_{n-1,1} & A_{n-1,2} & \dots & A_{n-1,n-1} & 0 \\
   A_{n,1} & A_{n,2} & \dots & A_{n,n-1} & A_{n,n} \\
 \end{pmatrix}.
\end{equation}
We denote by $\mathoo B^+(X)$ the class of operators induces by $\mathoo M^+$. Clearly, $\mathoo M^+$ is a closed sub\-algeb\-ra of the algebra $\mathoo M$. Therefore, $\mathoo B^+(X)$ is a closed subalgebra of the algebra $\mathoo B(X)$. We call operators from the class $\mathoo B^+(X)$ \emph{causal} in analogy with a similar class of operators in the control theory~\cite{Desoer-Vidyasagar,Feintuch-Saeks,Willems} and in the theory of functional differential equations~\cite{Kurbatov-SMZh75:eng,Kurbatov-81DE:eng,Kurbatov99}, see also the references therein.
Namely, if one interprets the indices $i=1,\dots,n$ as successive instants of time, then the triangularity of a matrix $A$ means that the value $(Ax)_i$ of the 'output' $Ax$ at any instant $i$ may depend only on values $x_j$ of the `input' $x$ at the previous instants $j\le i$.

\begin{remark}\label{r:B_c is not full}
The subalgebra $\mathoo M^+$ (and consequently, the subalgebra $\mathoo B^+(X)$) may be not full if the space $X$ is infinite-dimensional. We give a corresponding example. Let $X$ be the space $L_p(\mathbb R)$, $1\le p\le\infty$. We represent $X=L_p(\mathbb R)$ as $L_p(-\infty,0]\oplus L_p[0,\infty)$, where $L_p(-\infty,0]$ and $L_p[0,\infty)$ are the subspaces of functions that are equal to zero outside $(-\infty,0]$ and $[0,\infty)$ respectively. Clearly, the operator of delay $\bigl(Sx\bigr)(t)=x(t-1)$ is induced by a lower triangular matrix (thus it is causal), but the inverse operator $\bigl(S^{-1}x\bigr)(t)=x(t+1)$ is induced by an upper triangular matrix (thus $S^{-1}$ is not causal). Consequently, in contrast to the finite-dimensional case, the (ordinary) spectrum of a triangular matrix may be not the union of the spectra of its diagonal blocks, see Proposition~\ref{p:c-spectrum}. See a more detailed discussion of this phenomenon in~\cite{Han-Lee-Lee}.
\end{remark}

If an operator $T\in\mathoo B^+(X)$ is invertible and the inverse operator belongs to $\mathoo B^+(X)$, we say that $T$ is \emph{causally invertible}.
We call the spectrum of $T\in\mathoo B^+(X)$ in the algebra $\mathoo B^+(X)$ the \emph{causal spectrum} and denote it by $\sigma^+(T)$. Clearly,
\begin{equation*}
\sigma(T)\subseteq\sigma^+(T).
\end{equation*}
We denote by $\rho^+(T)$ the \emph{causal resolvent set} $\mathbb C\setminus\sigma^+(T)$.
The same terminology and notation will be used for matrices $M\in\mathoo M^+$.

We recall that an open set $D\subseteq\mathbb C$ is called \emph{simply-connected} if any simple closed curve in $D$ can be shrunk continuously to a point.

\begin{proposition}\label{p:simply-connected domain}
Let the domain $D\subseteq\mathbb C$ of an analytic function $f$ be simply-connected {\rm(}examples of such functions are $\exp_{\pm,\,t}$ and $g_t${\rm)}. Let\/ $T\in\mathoo B^+(X)$. Then $\sigma^+(T)\subset D$ provided $\sigma(T)\subset D$.
Thus the function $f(T)$ of a causal operator\/ $T$ is defined in algebras $\mathoo B(X)$ and $\mathoo B^+(X)$ simultaneously.
\end{proposition}
\begin{proof}
A possible difficulty can occur when the spectrum $\sigma(T)$ is contained in the domain $D$ of the definition of $f$, but $\sigma^+(T)\nsubseteq D$. Therefore the resolvent $(\lambda\mathbf1-A)^{-1}$ in integral~\eqref{e:def of f(A)} is defined in $\mathoo B(X)$, but it may not exist in $\mathoo B^+(X)$.

By Proposition~\ref{p:spectrum in a subalgebra}, the causal spectrum $\sigma^+(T)$ is the union of the ordinary spectrum $\sigma(T)$ and (possibly) some bounded components of the resolvent set $\rho(A)$. Since the domain $D$ of $f$ is simply-connected, bounded components of the resolvent set $\rho(A)$ are contained in the domain $D$, provided the spectrum $\sigma(T)$ itself is contained in the domain $D$.
\end{proof}

\begin{proposition}[{\rm\cite{Kurbatov-SMZh75:eng},~\cite[Proposition 2.1.7]{Kurbatov99}}]\label{p:c-spectrum}
The causal spectrum of a lower triangular matrix $\{\,T_{ij}\,\}$ {\rm(}and the causal spectrum of the corresponding operator{\rm)} is the union of the {\rm(}ordinary{\rm)} spectra $\sigma(T_{ii})$ of the diagonal blocks $T_{ii}$.
\end{proposition}
\begin{proof} It suffices to prove that a lower triangular matrix has a lower triangular inverse if and only if all diagonal blocks $T_{ii}$ are invertible.

Let the lower triangular matrix $\{\,B_{ij}\,\}$ be the inverse of the lower triangular matrix~$\{\,T_{ij}\,\}$. Then it follows from the matrix multiplication rule that $B_{ii}$ are inverses of $T_{ii}$.

Conversely, let the diagonal blocks $T_{ii}$ be invertible.
Then from the Gaussian elimination algorithm it easily follows that the inverse matrix exists and is triangular.
\end{proof}

\section{Functions of block triangular matrices}\label{s:func of block triang matr}

\begin{theorem}\label{t:inv block triang matr}
Let a causal matrix
\begin{equation*}
T=\begin{pmatrix}
   T_{1,1} & 0 & \dots & 0 & 0 \\
   T_{2,1} & T_{2,2} & \dots & 0 & 0 \\
   \dots & \dots & \dots & \dots & \dots \\
   T_{n-1,1} & T_{n-1,2} & \dots & T_{n-1,n-1} & 0 \\
   T_{n,1} & T_{n,2} & \dots & T_{n,n-1} & T_{n,n}
 \end{pmatrix}
\end{equation*}
be causally invertible. Then the elements of the inverse matrix
\begin{equation*}
B=\begin{pmatrix}
   B_{1,1} & 0 & \dots & 0 & 0 \\
   B_{2,1} & B_{2,2} & \dots & 0 & 0 \\
   \dots & \dots & \dots & \dots & \dots \\
   B_{n-1,1} & B_{n-1,2} & \dots & B_{n-1,n-1} & 0 \\
   B_{n,1} & B_{n,2} & \dots & B_{n,n-1} & B_{n,n}
 \end{pmatrix}
\end{equation*}
have the form
\begin{equation*}
B_{i,j}=\sum_{i=i_1>i_2\dots>i_m=j}(-1)^{m+1}T_{i_1,i_1}^{-1}T_{i_1,i_2}T_{i_2,i_2}^{-1}T_{i_2,i_3}\dots
T_{i_{m-1},i_m}T_{i_m,i_m}^{-1},\qquad i\ge j.
\end{equation*}
In particular,
\begin{align*}
B_{i,i}&=T_{i,i}^{-1},\\
B_{i+1,i} &=-T_{i+1,i+1}^{-1}T_{i+1,i}T_{i,i}^{-1},\\
B_{i+2,i} &=-T_{i+2,i+2}^{-1}T_{i+2,i}T_{i,i}^{-1}
+T_{i+2,i+2}^{-1}T_{i+2,i+1}T_{i+1,i+1}^{-1}T_{i+1,i}T_{i,i}^{-1}.
\end{align*}
\end{theorem}
\begin{proof}
Let us verify that $TB=\mathbf1$. Clearly, $(TB)_{ii}=\mathbf1$. We calculate, for example, $(TB)_{n,1}$:
\begin{align*}
(TB)_{n,1}&=T_{n,1}T_{11}^{-1}-T_{n,2}T_{22}^{-1}T_{21}T_{11}^{-1}
+T_{n,3}(-T_{33}^{-1}T_{31}T_{11}^{-1}+T_{33}^{-1}T_{32}T_{22}^{-1}T_{21}T_{11}^{-1})+\dots\\
&+T_{n,n}\sum_{n=i_1>i_2\dots>i_m=1}(-1)^{m+1}T_{n,n}^{-1}T_{n,i_2}T_{i_2,i_2}^{-1}T_{i_2,i_3}\dots
T_{i_{m-1},1}T_{1,1}^{-1}\\
&=T_{n,1}T_{11}^{-1}-T_{n,2}T_{22}^{-1}T_{21}T_{11}^{-1}
+T_{n,3}(-T_{33}^{-1}T_{31}T_{11}^{-1}+T_{33}^{-1}T_{32}T_{22}^{-1}T_{21}T_{11}^{-1})+\dots\\
&+\sum_{n=i_1>i_2\dots>i_m=1}(-1)^{m+1}T_{n,i_2}T_{i_2,i_2}^{-1}T_{i_2,i_3}\dots
T_{i_{m-1},1}T_{1,1}^{-1}=0.
\end{align*}
In a similar way one establishes that $BT=\mathbf1$.
\end{proof}

\begin{theorem}\label{t:res block triang matr}
Let $A\in\mathoo M^+$. Then the causal resolvent {\rm(}i.e., the resolvent $(\lambda\mathbf1-A)^{-1}$ at the points $\lambda\in\rho^+(A)${\rm)} of $A$ has the form
\begin{equation*}
(\lambda\mathbf1-A)^{-1}=\begin{pmatrix}
R_{1,1} & 0 & \dots & 0 & 0 \\
R_{2,1} & R_{2,2} & \dots & 0 & 0 \\
   \dots & \dots & \dots & \dots & \dots \\
   R_{n-1,1} & R_{n-1,2} & \dots & R_{n-1,n-1} & 0 \\
   R_{n,1} & R_{n,2} & \dots & R_{n,n-1} & R_{n,n} \\
 \end{pmatrix},
\end{equation*}
where $R_{ij}$ for $i\ge j$ are defined by the formula
\begin{equation*}
R_{ij}=\sum_{i=i_1>i_2\dots>i_m=j}
(\lambda\mathbf1-A_{i_1,i_1})^{-1}A_{i_1,i_2}(\lambda\mathbf1-A_{i_2,i_2})^{-1}A_{i_2,i_3}\dots A_{i_{m-1},i_m}(\lambda\mathbf1-A_{i_m,i_m})^{-1}.
\end{equation*}
In particular,
\begin{align*}
R_{ii}&=(\lambda\mathbf1-A_{ii})^{-1},\\
R_{i+1,i}&=(\lambda\mathbf1-A_{i+1,i+1})^{-1}A_{i+1,i}(\lambda\mathbf1-A_{ii})^{-1},\\
R_{i+2,i}&=(\lambda\mathbf1-A_{i+2,i+2})^{-1}A_{i+2,i}(\lambda\mathbf1-A_{ii})^{-1}\\
&+(\lambda\mathbf1-A_{i+2,i+2})^{-1}A_{i+2,i+1}(\lambda\mathbf1-A_{i+1,i+1})^{-1}A_{i+1,i}(\lambda\mathbf1-A_{ii})^{-1}.
\end{align*}
\end{theorem}
\begin{proof}
The proof follows from Theorem~\ref{t:inv block triang matr}. The sign $(-1)^{m+1}$ disappears because $(\lambda\mathbf1-A)_{ij}=-A_{ij}$ for $i>j$.
\end{proof}

\begin{theorem}\label{t:func block triang matr}
Let a function $f$ be analytic in a neighborhood of the causal spectrum $\sigma^+(A)$ of a matrix $A\in\mathoo M^+$. Then the matrix $F=f(A)$ has the form
\begin{equation*}
F=\begin{pmatrix}
F_{1,1} & 0 & \dots & 0 & 0 \\
F_{2,1} & F_{2,2} & \dots & 0 & 0 \\
   \dots & \dots & \dots & \dots & \dots \\
   F_{n-1,1} & F_{n-1,2} & \dots & F_{n-1,n-1} & 0 \\
   F_{n,1} & F_{n,2} & \dots & F_{n,n-1} & F_{n,n} \\
 \end{pmatrix},
\end{equation*}
where $F_{ij}$ for $i\ge j$ are defined by the formula
\begin{align*}
F_{ij}&=\sum_{i=i_1>i_2\dots>i_m=j}\frac1{2\pi i}\int_\Gamma f(\lambda)
(\lambda\mathbf1-A_{i_1,i_1})^{-1}A_{i_1,i_2}(\lambda\mathbf1-A_{i_2,i_2})^{-1}A_{i_2,i_3}\dots\\
&\times A_{i_{m-1},i_m}(\lambda\mathbf1-A_{i_m,i_m})^{-1}\,d\lambda,
\end{align*}
where $\Gamma$ surrounds the causal spectrum $\sigma^+(A)$ of the matrix $A$.
In particular,
\begin{align*}
F_{ii}&=\int_\Gamma f(\lambda)(\lambda\mathbf1-A_{ii})^{-1}\,d\lambda,\\
F_{i+1,i}
&=\int_\Gamma f(\lambda)(\lambda\mathbf1-A_{i+1,i+1})^{-1}A_{i+1,i}(\lambda\mathbf1-A_{ii})^{-1}\,d\lambda,\\
F_{i+2,i}
&=\int_\Gamma f(\lambda)(\lambda\mathbf1-A_{i+2,i+2})^{-1}A_{i+2,i}(\lambda\mathbf1-A_{ii})^{-1}\,d\lambda\\
&+\int_\Gamma f(\lambda)(\lambda\mathbf1-A_{i+2,i+2})^{-1}A_{i+2,i+1}
(\lambda\mathbf1-A_{i+1,i+1})^{-1}A_{i+1,i}(\lambda\mathbf1-A_{ii})^{-1}\,d\lambda.
\end{align*}
\end{theorem}
\begin{proof}
Substituting into formula~\eqref{e:def of f(A)} the representation of $(\lambda\mathbf1-A)^{-1}$ from Theorem~\ref{t:res block triang matr}, we obtain the desired result.
\end{proof}

\begin{remark}\label{r:Carbonell-Jimenez-Pedroso}
Let a matrix $A\in\mathoo M^+$ has only two non-zero diagonals:
\begin{equation*}
\begin{pmatrix}
   A_{1,1} & 0 & \dots & 0 & 0 \\
   A_{2,1} & A_{2,2} & \dots & 0 & 0 \\
   \dots & \dots & \dots & \dots & \dots \\
   0 & 0 & \dots & A_{n-1,n-1} & 0 \\
   0 & 0 & \dots & A_{n,n-1} & A_{n,n} \\
 \end{pmatrix}.
\end{equation*}
Let a function $f$ be analytic in a neighborhood of the causal spectrum $\sigma^+(A)$ of the matrix $A$.
Then it follows from Theorem~\ref{t:func block triang matr} that the elements $F_{ij}$ for $i\ge j$ of the matrix $F=f(A)$ consist of exactly one summand:
\begin{equation*}
F_{ij}=\frac1{2\pi i}\int_\Gamma f(\lambda)
(\lambda\mathbf1-A_{i,i})^{-1}A_{i,i+1}(\lambda\mathbf1-A_{i+1,i+1})^{-1}A_{i+1,i+2}\dots
A_{j-1,j}(\lambda\mathbf1-A_{j,j})^{-1}\,d\lambda.
\end{equation*}
For the function $f=\exp_{+,\,t}$ this phenomenon was described in~\cite{Carbonell-Jimenez-Pedroso08} and applied in~\cite{Goodwin-Kuprov15}.
\end{remark}

\begin{corollary}\label{c:g of block triang matr}
Let the domain $D\subseteq\mathbb C$ of an analytic function $f$ be simply-connected {\rm(}examples of such functions are $\exp_{\pm,\,t}$ and $g_t${\rm)}. Then the conclusion of Theorem~\ref{t:func block triang matr} is true if the function $f$ is analytic in a neighborhood of the ordinary spectrum $\sigma(A)$ of the matrix $A\in\mathoo M^+$.
\end{corollary}
\begin{proof}
The proof follows from Proposition~\ref{p:simply-connected domain}.
\end{proof}

\begin{remark}\label{r:func block triang matr}
For scalar matrices, Theorem~\ref{t:func block triang matr} goes back to~\cite{Rosenbloom1967}. 
For matrices consisting of finite-dimensional blocks, it was published in~\cite[Theorem 2]{Davis_Ch73}. More precisely, in~\cite{Davis_Ch73} it was considered only the case of polynomial functions $f$; but from the case of polynomials, it follows the case of general analytic functions, since (if $A$ is a scalar matrix) any analytic function can be replaced by its interpolating polynomial.
\end{remark}

\section{Divided differences}\label{s:divided differences}

Let $\mu_1$, $\mu_2$, \dots, $\mu_n$ be given complex numbers (some of them may coincide
with others) called \emph{points of interpolation}. Let a complex-valued function $f$ be
defined and analytic in a neighborhood of these points. \emph{Divided differences}
of the function $f$ with respect to the points $\mu_1$,
$\mu_2$, \dots, $\mu_n$ are defined (see, e.g.,~\cite{
Gelfond:eng,Jordan}) by the recurrent relations
 \begin{equation}\label{e:divided differences}
 \begin{split}
f^{[0]}(\mu_i)&=f(\mu_i),\\
f^{[1]}(\mu_i,\mu_{i+1})&=\frac{f^{[0]}(\mu_{i+1})-f^{[0]}(\mu_i)}{\mu_{i+1}-\mu_i},\\
f^{[m]}(\mu_i,\dots,\mu_{i+m})&=\frac{f^{[m-1]}(\mu_{i+1},\dots,\mu_{i+m})
-f^{[m-1]}(\mu_{i},\dots,\mu_{i+m-1})} {\mu_{i+m}-\mu_i}.
 \end{split}
 \end{equation}
In these formulas, if the denominator vanishes, then the quotient is understood as the derivative with respect to one of the arguments of the previous divided difference (the naturalness of this agreement can be derived by continuity from Corollary~\ref{c:f[] is continuous}).

\begin{proposition}[{\rm\cite[ch.~1, formula (54)]{Gelfond:eng}}]\label{p:f[] via Gamma}
Let a function $f$ be analytic in an open neighbourhood of the points of interpolation $\mu_1$, $\mu_2$, \dots, $\mu_n$. Then divided differences admit the representation
\begin{equation*}
f^{[n-1]}(\mu_{1},\dots,\mu_{n})=\frac1{2\pi i}\int_{\Gamma}\frac{f(z)}{\Omega(z)}\,dz,
\end{equation*}
where the contour $\Gamma$ encloses all the points of interpolation and
\begin{equation*}
\Omega(z)=\prod_{k=1}^n(z-\mu_k).
\end{equation*}
\end{proposition}

\begin{corollary}\label{c:f[] is continuous}
Divided differences are differentiable functions of their arguments.
\end{corollary}
\begin{proof}
The proof follows from Proposition~\ref{p:f[] via Gamma}.
\end{proof}

\begin{corollary}\label{c:domain of f[]}
If $D\subseteq\mathbb C$ is the domain of an analytic function $f$, then $f^{[n-1]}$ is defined in $D^n$.
\end{corollary}
\begin{proof}
The proof follows from Proposition~\ref{p:f[] via Gamma}.
\end{proof}

 \begin{corollary}\label{c:f[] is indep of order}
Divided differences $f^{[n-1]}(\mu_{1},\dots,\mu_{n})$ are
symmetric function, i.e., they do not depend on the order of their arguments $\mu_1$, \dots, $\mu_{n}$.
 \end{corollary}
 \begin{proof}
The proof follows from Proposition~\ref{p:f[] via Gamma}.
 \end{proof}

\begin{proposition}[{\rm\cite[ch.~1, formula (48)]{Gelfond:eng},~\cite[p.~19, formula~(1)]{Jordan}}]\label{p:repr of Delta}
Let the points of interpolation $\mu_{1},\dots,\mu_{n}$ be distinct. Then divided differences admit the representation
\begin{equation*}
f^{[n-1]}(\mu_{1},\dots,\mu_{n})=\sum_{j=1}^n\frac{f(\mu_j)}{\prod\limits_{\substack{k=1\\k\neq j}}^n(\mu_j-\mu_k)}.
\end{equation*}
\end{proposition}
 \begin{proof}
The proof follows from Proposition~\ref{p:f[] via Gamma}.
 \end{proof}

For a function $f$ analytic in a neighborhood of the causal spectrum of a matrix $A\in\mathoo M^+$, we denote by $f^{[m-1]}(A;i_1,i_2,\dots,i_m)$ the summands from Theorem~\ref{t:func block triang matr}:
\begin{multline}\label{e:def of f[m-1]}
f^{[m-1]}(A;i_1,i_2,\dots,i_m)
=\frac1{2\pi i}\int_\Gamma f(\lambda)
(\lambda\mathbf1-A_{i_1,i_1})^{-1}A_{i_1,i_2}(\lambda\mathbf1-A_{i_2,i_2})^{-1}A_{i_2,i_3}\dots\\
\times
A_{i_{m-1},i_m}(\lambda\mathbf1-A_{i_m,i_m})^{-1}\,d\lambda,
\end{multline}
where $\Gamma$ surrounds the causal spectrum $\sigma^+(A)$ of the matrix $A$.

\begin{theorem}\label{t:boxdot=boxtimes}
Let a function\/ $f$ be analytic in a neighborhood of the causal spectrum $\sigma^+(A)$ of a matrix $A\in\mathoo M^+$. Then
\begin{equation}\label{e:def of f[m-1]}
\begin{split}
f^{[m-1]}(A&;i_1,i_2,\dots,i_m)
=\frac1{(2\pi i)^m}\int_{\Gamma_{i_1}}\dots\int_{\Gamma_{i_m}}f^{[m-1]}(\lambda_{i_1},\dots,\lambda_{i_m})
(\lambda_{i_1}\mathbf1-A_{i_1,i_1})^{-1}\\
&\times A_{i_1,i_2}
(\lambda_{i_2}\mathbf1-A_{i_2,i_2})^{-1}A_{i_2,i_3}
\dots
A_{i_{m-1},i_m}(\lambda_{i_m}\mathbf1-A_{i_m,i_m})^{-1}\,d\lambda_{i_1}\dots\,d\lambda_{i_m},
\end{split}
\end{equation}
where $\Gamma_{i_k}$ surrounds the spectrum of $A_{i_k,i_k}$.
\end{theorem}
\begin{proof}
Since $f$ is analytic in a neighborhood of $\sigma^+(A)=\bigcup_{i=1}^n\sigma(A_{ii})$ (see Proposition~\ref{p:c-spectrum}), we may assume without loss of generality that $\Gamma_{i_k}$ in~\eqref{e:def of f[m-1]} surrounds the whole $\bigcup_{i=1}^n\sigma(A_{ii})$ and, moreover, $\Gamma_{i_k}$ surrounds $\Gamma_{i_{k+1}}$, see Figure~\ref{f:Gamma_i}.

 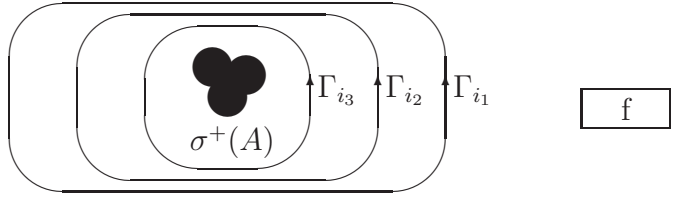
\begin{figure}[htb]
\unitlength=1.mm
\begin{center}
\begin{picture}(50,30)
\put(22.5,18){\circle*{10}} \put(20,15){\circle*{10}} \put(18,19){\circle*{10}}
\put(20,15){\oval(22,19)} \put(31,15){\vector(0,1){3}} \put(32,15){$\Gamma_{i_3}$}
\put(20,15){\oval(40,22)} \put(40,15){\vector(0,1){3}} \put(41,15){$\Gamma_{i_2}$}
\put(20,15){\oval(58,25)} \put(49,15){\vector(0,1){3}} \put(50,15){$\Gamma_{i_1}$}
\put(15,8){$\sigma^+(A)$} \put(67,12){\boxed{$\phantom{aa}f\phantom{aa}$}}
\end{picture}\hfil
\caption{The choice of the contours $\Gamma_{i_k}$ in the proof of Theorem~\ref{t:boxdot=boxtimes}. The symbol $\boxed{f}$ means the localization of singularities of $f$ }\label{f:Gamma_i}
\end{center}
 \end{figure}

Since the contours $\Gamma_{i_k}$ are disjoint, the points $\lambda_{i_1}$, \dots, $\lambda_{i_m}$ in the integrand of~\eqref{e:def of f[m-1]} are distinct. Therefore we can substitute the representation of divided differences from Proposition~\ref{p:repr of Delta} into definition~\eqref{e:def of f[m-1]}:
\begin{equation}\label{e:boxed of f[m]}
\begin{split}
f^{[m-1]}(A&;i_1,i_2,\dots,i_m)
=\frac1{(2\pi i)^m}\int_{\Gamma_{i_1}}\dots\int_{\Gamma_{i_m}}
\sum_{j=1}^{m}\frac{f(\lambda_{i_j})}{\prod\limits_{\substack{k=1\\k\neq j}}^m(\lambda_{i_j}-\lambda_{i_k})}
\\
&\times(\lambda_{i_1}\mathbf1-A_{i_1,i_1})^{-1}A_{i_1,i_2}
\dots
A_{i_{m-1},i_m}(\lambda_{i_m}\mathbf1-A_{i_m,i_m})^{-1}\,d\lambda_{i_1}\dots\,d\lambda_{i_m}.
\end{split}
\end{equation}
Let us begin with the first summand. We have
\begin{align*}
\frac1{2\pi i}
&\int_{\Gamma_{i_{i_1}}}f(\lambda_{i_1})(\lambda_{i_1}\mathbf1-A_{i_1,i_1})^{-1}A_{i_1,i_2}\\
&\times\biggl[\dots
\frac1{2\pi i}\int_{\Gamma_{i_{m-2}}}\frac{1}{\lambda_{i_1}-\lambda_{i_{m-2}}}
\times(\lambda_{i_{m-2}}\mathbf1-A_{i_{m-2},i_{m-2}})^{-1}A_{i_{m-2},i_{m-1}}\\
&\times\biggl[\frac1{2\pi i}\int_{\Gamma_{i_{m-1}}}\frac{1}{\lambda_{i_1}-\lambda_{i_{m-1}}}
(\lambda_{i_{m-1}}\mathbf1-A_{i_{m-1},i_{m-1}})^{-1}A_{i_{m-1},i_m}\\
&\times\biggl[\frac1{2\pi i}\int_{\Gamma_{i_m}}\frac1{\lambda_{i_1}-\lambda_{i_m}}
(\lambda_{i_m}\mathbf1-A_{i_m,i_m})^{-1}\,d\lambda_{i_m}\biggr]
\,d\lambda_{i_{m-1}}\biggr]\,d\lambda_{i_{m-2}}\dots\biggr]\,d\lambda_{i_1}.
\end{align*}
By Corollary~\ref{c:r_0}, for the internal integral, we have
\begin{equation*}
\frac1{2\pi i}\int_{\Gamma_{i_m}}\frac1{\lambda_{i_1}-\lambda_{i_m}}
(\lambda_{i_m}\mathbf1-A_{i_m,i_m})^{-1}\,d\lambda_{i_m}
=(\lambda_{i_1}\mathbf1-A_{i_m,i_m})^{-1}.
\end{equation*}
Now we can calculate the next internal integral (again using Corollary~\ref{c:r_0}):
\begin{multline*}
\frac1{2\pi i}\int_{\Gamma_{i_{m-1}}}\frac{1}{\lambda_{i_1}-\lambda_{i_{m-1}}}
(\lambda_{i_{m-1}}\mathbf1-A_{i_{m-1},i_{m-1}})^{-1}A_{i_{m-1},i_m}
(\lambda_{i_{1}}\mathbf1-A_{i_m,i_m})^{-1}
\,d\lambda_{i_{m-1}}\\
=\biggl[\frac1{2\pi i}\int_{\Gamma_{i_{m-1}}}\frac{1}{\lambda_{i_1}-\lambda_{i_{m-1}}}
(\lambda_{i_{m-1}}\mathbf1-A_{i_{m-1},i_{m-1}})^{-1}\,d\lambda_{i_{m-1}}\biggr]A_{i_{m-1},i_m}
(\lambda_{i_1}\mathbf1-A_{i_m,i_m})^{-1}\\
=(\lambda_{i_1}\mathbf1-A_{i_{m-1},i_{m-1}})^{-1}A_{i_{m-1},i_m}
(\lambda_{i_1}\mathbf1-A_{i_m,i_m})^{-1}.
\end{multline*}
And so on. Finally, we arrive at the representation (for the first summand in~\eqref{e:boxed of f[m]})
\begin{equation*}
\frac1{2\pi i}\int_{\Gamma_{i_1}}f(\lambda_{i_1})
(\lambda_{i_1}\mathbf1-A_{i_1,i_1})^{-1}A_{i_1,i_2}(\lambda_{i_1}\mathbf1-A_{i_2,i_2})^{-1}A_{i_2,i_3}\dots
A_{i_{m-1},i_m}(\lambda_{i_1}\mathbf1-A_{i_m,i_m})^{-1}\,d\lambda_{i_1},
\end{equation*}
which coincides with formula~\eqref{e:def of f[m-1]}.

Next we show that the other summands in~\eqref{e:boxed of f[m]} are zero. Let us consider, for example, the second summand
\begin{align*}
\frac1{(2\pi i)^m}&\int_{\Gamma_{i_1}}\dots\int_{\Gamma_{i_m}}
\frac{f(\lambda_{i_2})}{\prod\limits_{\substack{k=1\\k\neq 2}}^m(\lambda_{i_2}-\lambda_{i_k})}
\\
&\times(\lambda_{i_1}\mathbf1-A_{i_1,i_1})^{-1}A_{i_1,i_2}
\dots
A_{i_{m-1},i_m}(\lambda_{i_m}\mathbf1-A_{i_m,i_m})^{-1}\,d\lambda_{i_1}\dots\,d\lambda_{i_m}.
\end{align*}
Proceeding as above (i.e. successively calculating integrals over all variables except $\lambda_{i_1}$), at the final stage, we arrive at the integral
\begin{equation*}
\frac1{2\pi i}\int_{\Gamma_{i_1}}\frac1{\lambda_{i_2}-\lambda_{i_1}}(\lambda_{i_1}\mathbf1-A_{i_1,i_1})^{-1}
\,d\lambda_{i_1}.
\end{equation*}
We notice that the singularity of the function $\lambda_{i_1}\mapsto\frac1{\lambda_{i_2}-\lambda_{i_1}}$ (i.e., the point $\lambda_{i_2}\in\Gamma_{i_2}$) lies inside the contour $\Gamma_{i_1}$. Hence the integrand $\lambda_{i_1}\mapsto\frac1{\lambda_{i_2}-\lambda_{i_1}}(\lambda_{i_1}\mathbf1-A_{i_1,i_1})^{-1}$ is analytic outside the contour $\Gamma_{i_1}$ and decreases at infinity as $\frac1{\lambda_{i_1}^2}$. Therefore the integral equals zero.
\end{proof}

\begin{remark}\label{r:Theorem 41}
For the case of the first divided difference, a more detailed discussion of formula~\eqref{e:def of f[m-1]} can be found in~\cite[Theorem 41]{Kurbatov-Kurbatova-Oreshina}, see also the references therein.
\end{remark}

\begin{corollary}\label{c:func block triang matr}
Let a function $f$ be analytic in a neighborhood of the causal spectrum $\sigma^+(A)$ of a matrix $A\in\mathoo M^+$. Then the matrix $F=f(A)$ has the form
\begin{equation*}
F=\begin{pmatrix}
F_{11} & 0 & \dots & 0 & 0 \\
F_{21} & F_{22} & \dots & 0 & 0 \\
   \dots & \dots & \dots & \dots & \dots \\
   F_{n-1,1} & F_{n-1,2} & \dots & F_{n-1,n-1} & 0 \\
   F_{n,1} & F_{n,2} & \dots & F_{n,n-1} & F_{n,n} \\
 \end{pmatrix},
\end{equation*}
where $F_{ij}$, $i\ge j$, admits the representation
\begin{equation*}
F_{ij}=\sum_{i=i_1>i_2\dots>i_m=j}f^{[m-1]}(A;i_1,i_2,\dots,i_m).
\end{equation*}
\begin{proof}
The proof follows from Theorems~\ref{t:func block triang matr} and~\ref{t:boxdot=boxtimes}.
\end{proof}
\end{corollary}

\begin{theorem}\label{t:exp of block tri mat}
Let $A\in\mathoo M^+$. Then
\begin{equation*}
\exp_{\pm,\,t}(A)=
\begin{pmatrix}
E_{\pm,1,1} & 0 & \dots & 0 & 0 \\
E_{\pm,2,1} & E_{\pm,2,2} & \dots & 0 & 0 \\
   \dots & \dots & \dots & \dots & \dots \\
   E_{\pm,n-1,1} & E_{\pm,n-1,2} & \dots & E_{\pm,n-1,n-1} & 0 \\
   E_{\pm,n,1} & E_{\pm,n,2} & \dots & E_{\pm,n,n-1} & E_{\pm,n,n} \\
 \end{pmatrix},
\end{equation*}
where $E_{\pm,i,j}$, $i\ge j$, admits the representation
\begin{equation*}
E_{\pm,i,j}=\sum_{i=i_1>i_2\dots>i_m=j}\exp_{\pm,\,t}^{[m-1]}(A;i_1,i_2,\dots,i_m);
\end{equation*}
and {\rm(}provided the spectrum $\sigma(A)$ does not intersect the imaginary axis{\rm)}
\begin{equation*}
g_t(A)=
\begin{pmatrix}
G_{1,1} & 0 & \dots & 0 & 0 \\
G_{2,1} & G_{2,2} & \dots & 0 & 0 \\
   \dots & \dots & \dots & \dots & \dots \\
   G_{n-1,1} & G_{n-1,2} & \dots & G_{n-1,n-1} & 0 \\
   G_{n,1} & G_{n,2} & \dots & G_{n,n-1} & G_{n,n} \\
 \end{pmatrix},
\end{equation*}
where $G_{i,j}$, $i\ge j$, admits the representation
\begin{equation*}
G_{i,j}=\sum_{i=i_1>i_2\dots>i_m=j}g_{t}^{[m-1]}(A;i_1,i_2,\dots,i_m).
\end{equation*}
\end{theorem}
\begin{proof}
The proof follows from Corollary~\ref{c:g of block triang matr} and Theorem~\ref{t:boxdot=boxtimes}.
\end{proof}

\section{Divided differences of the functions $\exp_{\pm,\,t}$ and $g_t$}\label{s:exp[n] and g[n]}

\begin{lemma}\label{l:r_nu}
Let the points of interpolation $\lambda_j$ be distinct.
Then the divided differences of the function $r_\lambda(\nu)=\frac1{\lambda-\nu}$ admit the representation
\begin{equation*}
r_\lambda^{[n-1]}(\lambda_1,\dots,\lambda_n)=\frac1{\prod_{j=1}^n(\lambda-\lambda_j)}.
\end{equation*}
\end{lemma}
\begin{proof}
The proof is by induction on $n$. For $n=1$ we have
\begin{equation*}
r_\lambda^{[1]}(\lambda_1,\lambda_2)
=\frac{\frac1{\lambda-\lambda_1}-\frac1{\lambda-\lambda_2}}{\lambda_1-\lambda_2}
=\frac1{(\lambda-\lambda_1)(\lambda-\lambda_2)}.
\end{equation*}
Assuming that the formula holds for $n-2$, we prove it for $n-1$. We have
\begin{align*}
r_\lambda^{[n-1]}(\lambda_1,\dots,\lambda_n)
=&\frac{r_\lambda^{[n-1]}(\lambda_1,\dots,\lambda_{n-2},\lambda_{n-1})
-r_\lambda^{[n-1]}(\lambda_1,\dots,\lambda_{n-2},\lambda_{n})
}{\lambda_{n-1}-\lambda_{n}}\\
=&\frac{\frac1{\prod_{j=1}^{n-2}(\lambda-\lambda_j)}\frac1{\lambda-\lambda_{n-1}}
-\frac1{\prod_{j=1}^{n-2}(\lambda-\lambda_j)}\frac1{\lambda-\lambda_n}
}{\lambda_{n-1}-\lambda_{n}}\\
=&\frac1{\prod_{j=1}^{n-2}(\lambda-\lambda_j)}
\frac{\frac1{\lambda-\lambda_{n-1}}-\frac1{\lambda-\lambda_n}}{\lambda_{n-1}-\lambda_{n}}\\
=&\frac1{\prod_{j=1}^{n}(\lambda-\lambda_j)}.\qed
\end{align*}
\renewcommand\qed{}
\end{proof}

We recall~\cite{Evgrafov-AF:eng,LePage,van_der_Pol-Bremmer} that the \emph{bilateral {\rm(}\textrm{or} two-sided{\rm)} Laplace transform} of a function $f:\,\mathbb R\to\mathbb C$ is the function
\begin{equation*}
\bigl(\mathcal B f\bigr)(\lambda)=\int_{-\infty}^\infty e^{-\lambda t}f(t)\,dt.
\end{equation*}
The value $\bigl(\mathcal B f\bigr)(\lambda)$ at the point $\lambda\in\mathbb C$ is defined if the integral converges absolutely.
If $f$ equals zero on $(-\infty,0)$ (as the function $\exp_{+,\,(\cdot)}$), this definition takes the form
\begin{equation*}
\bigl(\mathcal B f\bigr)(\lambda)=\int_{0}^\infty e^{-\lambda t}f(t)\,dt.
\end{equation*}
Usually in this case, the integral converges absolutely for $\Real\lambda$ sufficiently large.
If $f$ equals zero on $(0,\infty)$ (as the function $\exp_{-,\,(\cdot)}$), the definition of the bilateral Laplace transform takes the form
\begin{equation*}
\bigl(\mathcal B f\bigr)(\lambda)=\int_{-\infty}^0 e^{-\lambda t}f(t)\,dt.
\end{equation*}
In this case, we assume
that the integral converges absolutely for $\Real\lambda$ sufficiently small.

We recall the following statement.

\begin{lemma}\label{l:Laplace of exp_pm and g_t}
Let $\lambda_0\in\mathbb C$.
\noindent
\begin{enumerate}
 \item[{\rm(a)}] The bilateral Laplace transform of the function $t\mapsto\exp_{+,\,t}(\lambda_0)$ is the function
\begin{equation*}
\bigl(\mathcal B\,\exp_{+,\,(\cdot)}(\lambda_0)\bigr)(\lambda)=\frac1{\lambda-\lambda_0},\qquad\Real\lambda>\Real\lambda_0.
\end{equation*}
 \item[{\rm(b)}] The bilateral Laplace transform of the function $t\mapsto\exp_{-,\,t}(\lambda_0)$ is the function
\begin{equation*}
\bigl(\mathcal B\,\exp_{-,\,(\cdot)}(\lambda_0)\bigr)(\lambda)=\frac1{\lambda-\lambda_0},
\qquad\Real\lambda<\Real\lambda_0.
\end{equation*}
 \item[{\rm(c)}] The bilateral Laplace transform of the function $t\mapsto g_{t}(\lambda_0)$, $\Real\lambda_0\neq0$, is the function
     {\rm(}the complete domain of the function of $\mathcal B\,g_{(\cdot)}(\lambda_0)$ is $\Real\lambda>\Real\lambda_0$ if $\Real\lambda_0<0$ and is $\Real\lambda<\Real\lambda_0$ if $\Real\lambda_0>0${\rm)}
\begin{equation*}
\bigl(\mathcal B\,g_{(\cdot)}(\lambda_0)\bigr)(\lambda)=\frac1{\lambda-\lambda_0},\qquad|\Real\lambda|<|\Real\lambda_0|.
\end{equation*}
\end{enumerate}
\end{lemma}
\begin{proof}
Assertion (a) is widely known~\cite[pp.~300, 305]{LePage}.
The proofs of all assertions are reduced to straightforward calculations.
The proof of assertion (c) can also be obtained from the definition of $g_t$ and (a) and (b).
\end{proof}

We recall~\cite[ch.~6]{Rudin-FA:eng} that the \emph{convolution} of two summable functions $f,g:\,\mathbb R\to\mathbb C$ is the function
\begin{equation*}
\bigl(f*g\bigr)(t)=\int_{-\infty}^{\infty}f(s)g(t-s)\,ds.
\end{equation*}
If $f(t)=0$ and $g(t)=0$ for $t<0$, then this formula takes the form
\begin{equation*}
\bigl(f*g\bigr)(t)=
\begin{cases}
\int_{0}^{t}f(s)g(t-s)\,ds& \text{for $t>0$},\\
0& \text{for $t<0$}.
\end{cases}
\end{equation*}
If $f(t)=0$ and $g(t)=0$ for $t>0$, then the definition of convolution takes the form
\begin{equation*}
\bigl(f*g\bigr)(t)=
\begin{cases}
0& \text{for $t>0$},\\
\int_{t}^{0}f(s)g(t-s)\,ds& \text{for $t<0$}.
\end{cases}
\end{equation*}

\begin{theorem}\label{t:exp_+[n] and g_t[n]}
\noindent
\begin{enumerate}
 \item[{\rm(a)}] The divided differences of the function $t\mapsto\exp_{+,\,t}$ admit the representation
\begin{equation*}
\exp_{+,\,(\cdot)}^{[n-1]}(\lambda_1,\dots,\lambda_n)
=\exp_{+,\,(\cdot)}(\lambda_1)*\dots*\exp_{+,\,(\cdot)}(\lambda_n).
\end{equation*}
For example, for $t>0$, we have
\begin{align*}
\exp_{+,\,t}^{[1]}(\lambda_1,\lambda_2)&=\int_0^t\exp_{+,\,s}(\lambda_1)\exp_{+,\,t-s}(\lambda_2)\,ds,\\
\exp_{+,\,t}^{[2]}(\lambda_1,\lambda_2,\lambda_3)&
=\int_0^t\int_0^{r}\exp_{+,\,s}(\lambda_1)\exp_{+,\,r-s}(\lambda_2)\exp_{+,\,t-r}(\lambda_3)\,ds\,dr.
\end{align*}
 \item[{\rm(b)}] The divided differences of the function $t\mapsto\exp_{-,\,t}$ admit the representation
\begin{equation*}
\exp_{-,\,(\cdot)}^{[n-1]}(\lambda_1,\dots,\lambda_n)
=\exp_{-,\,(\cdot)}(\lambda_1)*\dots*\exp_{-,\,(\cdot)}(\lambda_n).
\end{equation*}
For example, for $t<0$, we have
\begin{align*}
\exp_{-,\,t}^{[1]}(\lambda_1,\lambda_2)&=\int_{t}^0\exp_{-,\,s}(\lambda_1)\exp_{-,\,t-s}(\lambda_2)\,ds,\\
\exp_{-,\,t}^{[2]}(\lambda_1,\lambda_2,\lambda_3)&
=\int_{t}^0\int_{r}^0\exp_{-,\,s}(\lambda_1)\exp_{-,\,r-s}(\lambda_2)\exp_{-,\,t-r}(\lambda_3)\,ds\,dr.
\end{align*}
 \item[{\rm(c)}] The divided differences of the function $t\mapsto g_{t}$ admit the representation
\begin{equation*}
g_{(\cdot)}^{[n-1]}(\lambda_1,\dots,\lambda_n)
=g_{(\cdot)}(\lambda_1)*\dots*g_{(\cdot)}(\lambda_n),
\qquad\Real\lambda_1,\dots,\Real\lambda_n\neq0.
\end{equation*}
For example, for $t\neq0$, we have
\begin{align*}
g_{t}^{[1]}(\lambda_1,\lambda_2)&=\int_{-\infty}^{\infty}g_{s}(\lambda_1)g_{t-s}(\lambda_2)\,ds,
\\
g_{t}^{[2]}(\lambda_1,\lambda_2,\lambda_3)
&=\int_{-\infty}^{\infty}\int_{-\infty}^{\infty}g_{s}(\lambda_1)g_{r-s}(\lambda_2)g_{t-r}(\lambda_3)\,ds\,dr.\notag
\end{align*}
\end{enumerate}
\end{theorem}
\begin{remark}\label{r:DD}
Assertion (a) is established in~\cite{Van_Loan78}.
Assertion (b) is proved in a similar way.
Assertion (c) is proved in~\cite[p.~53]{Godunov94:ODE:eng} for other aims. For completeness, we give here an independent proof of (c).
\end{remark}

\begin{proof}
Suppose that the points of interpolation $\lambda_j$ are distinct. By Proposition~\ref{p:repr of Delta}, we have
\begin{equation*}
g_t^{[n-1]}(\lambda_1,\dots,\lambda_n)
=\sum_{j=1}^n\frac{g_{t}(\lambda_j)}{\prod\limits_{\substack{k=1\\k\neq j}}^n(\lambda_j-\lambda_k)}.
\end{equation*}
Form this representation and Lemma~\ref{l:Laplace of exp_pm and g_t}, it easily follows that the bilateral Laplace transform of the function $t\mapsto g_t^{[n-1]}(\lambda_1,\dots,\lambda_n)$ is
\begin{equation*}
\bigl(\mathcal B\,g_{(\cdot)}^{[n-1]}(\lambda_1,\dots,\lambda_n)\bigr)(\lambda)
=\sum_{j=1}^n\frac1{\lambda-\lambda_j}\frac{1}{\prod\limits_{\substack{k=1\\k\neq j}}^n(\lambda_j-\lambda_k)}.
\end{equation*}
By Proposition~\ref{p:repr of Delta}, the last expression is the $(n-1)$-th divided difference of the function $r_\lambda(\nu)=\frac1{\lambda-\nu}$. By Lemma~\ref{l:r_nu},
\begin{equation*}
r_\lambda^{[n-1]}(\lambda_1,\dots,\lambda_n)=\frac1{\prod_{j=1}^n(\lambda-\lambda_j)}.
\end{equation*}

We apply the inverse bilatiral Laplace transform to tht function $\lambda\mapsto r_\lambda^{[n-1]}(\lambda_1,\dots,\lambda_n)$.
Clearly, the restriction of the bilatiral Laplace transform to the imaginary axis is the Fourier transform. The Fourier transform maps the convolution of functions to the product of their images~\cite[p.~337]{LePage}, which implies assertion (c).

The case of coinciding points of interpolation $\lambda_j$ follows from continuity.
\end{proof}

\section{The divided differences $\exp_{\pm,\,t}^{[m]}$ and $g_t^{[m]}$ with operator arguments}\label{s:exp and g of A}
In this Section, we apply previous results to the calculation of the functions $\exp_{\pm,\,t}^{[m]}$ and $g_t^{[m]}$ with operator arguments.

\begin{theorem}\label{t:g_t[n]}
Let $A\in\mathoo M^+$. Then for $t>0$, we have
\begin{align*}
\exp_{+,\,t}^{[m-1]}&(A;i_1,i_2,\dots,i_{m-1},i_m)
=\int_0^t\int_0^{s_{m-1}}\dots\int_0^{s_2}
\exp_{+,\,s_1}(A_{i_1,i_1})A_{i_1,i_2}\exp_{+,\,s_2-s_1}(A_{i_2,i_2})\\
&\times A_{i_2,i_3}\dots
\exp_{+,\,s_{m-1}-s_{m-2}}(A_{i_{m-1},i_{m-1}}) A_{i_{m-1},i_m}\exp_{+,\,t-s_{m-1}}(A_{i_m,i_m})
\,ds_1\,\dots\,ds_{m-1};\\
\intertext{for $t<0$, we have}
\exp_{-,\,t}^{[m-1]}&(A;i_1,i_2,\dots,i_{m-1},i_m)
=\int_t^0\int_{s_{m-1}}^0\dots\int_{s_2}^0
\exp_{-,\,s_1}(A_{i_1,i_1})A_{i_1,i_2}\exp_{-,\,s_2-s_1}(A_{i_2,i_2})\\
&\times A_{i_2,i_3}\dots
\exp_{-,\,s_{m-1}-s_{m-2}}(A_{i_{m-1},i_{m-1}}) A_{i_{m-1},i_m}\exp_{-,\,t-s_{m-1}}(A_{i_m,i_m})
\,ds_1\,\dots\,ds_{m-1};\\
\intertext{and {\rm(}if the spectrum $\sigma(A)$ is disjoint from the imaginary axis{\rm)} for $t\neq0$, we have}
g_t^{[m-1]}&(A;i_1,i_2,\dots,i_{m-1},i_m)
=\int_{-\infty}^{\infty}\dots\int_{-\infty}^{\infty}
g_{s_1}(A_{i_1,i_1})A_{i_1,i_2}\;g_{s_2-s_1}(A_{i_2,i_2})\\
&\times A_{i_2,i_3}\dots g_{s_{m-1}-s_{m-2}}(A_{i_{m-1},i_{m-1}})A_{i_{m-1},i_m}\;g_{t-s_{m-1}}(A_{i_m,i_m})
\,ds_1\,\dots\,ds_{m-1}.
\end{align*}
\end{theorem}
\begin{proof}
For simplicity of notation, we prove only the formula
\begin{equation*}
g_t^{[2]}(A;3,2,1)
=\int_{-\infty}^{\infty}\int_{-\infty}^{\infty}
g_{s_1}(A_{33})A_{32}\;g_{s_2-s_1}(A_{22})A_{21}\;g_{t-s_2}(A_{11})\,ds_1\,ds_2.
\end{equation*}
By Theorem~\ref{t:boxdot=boxtimes}, we have
\begin{multline*}
g_{t}^{[2]}(A;3,2,1)
=\frac1{(2\pi i)^3}\int_{\Gamma_1}\int_{\Gamma_2}\int_{\Gamma_3} g_{t}^{[2]}(\lambda_1,\lambda_2,\lambda_3)\\
\times
(\lambda_3\mathbf1-A_{33})^{-1}A_{32}
(\lambda_2\mathbf1-A_{22})^{-1}A_{21}
(\lambda_1\mathbf1-A_{11})^{-1}\,d\lambda_1\,d\lambda_2\,d\lambda_3.
\end{multline*}
By Theorem~\ref{t:exp_+[n] and g_t[n]}, we have
\begin{equation*}
g_{t}^{[2]}(\lambda_1,\lambda_2,\lambda_3)
=\int_{-\infty}^{\infty}\int_{-\infty}^{\infty}g_{s}(\lambda_1)\;g_{r-s}(\lambda_2)\;g_{t-r}(\lambda_3)\,ds\,dr.
\end{equation*}
Substituting the latter formula into the former one and performing the integration over $\lambda_1$, $\lambda_2$, and $\lambda_3$, we arrive at the desired formula.
\end{proof}

Combining Theorems~\ref{t:exp of block tri mat} and~\ref{t:g_t[n]}, we obtain the following examples.

\begin{example}\label{ex:exp of block triang matr}
Let $A$ be the block matrix
\begin{equation}\label{e:matrix A}
A=\begin{pmatrix}
 A_{11} & 0 & 0 \\
 A_{21} & A_{22} & 0 \\
 A_{31} & A_{32} & A_{33} \\
 \end{pmatrix}.
\end{equation}
Then for $t>0$, we have
\begin{equation}\label{e:exp(A) via exp[i]}
\exp_{+,\,t}(A)=\begin{pmatrix}
\exp_{+,\,t}(A_{11}) & 0 & 0 \\
\exp_{+,\,t}^{[1]}(A;2,1) & \exp_{+,\,t}(A_{22}) & 0 \\
\exp_{+,\,t}^{[1]}(A;3,1)
+\exp_{+,\,t}^{[2]}(A;3,2,1)
   &\exp_{+,\,t}^{[1]}(A;3,2) & \exp_{+,\,t}(A_{33}) \\
 \end{pmatrix},
\end{equation}
where
\begin{align}
\exp_{+,\,t}^{[1]}(A;i_1,i_2)
&=\int_0^t\exp_{+,\,s}(A_{i_,1i_1})A_{i_1,i_2}\exp_{+,\,t-s}(A_{i_2,i_2})\,ds,\label{e:exp[1] from Bellman} \\
\exp_{+,\,t}^{[2]}(A;3,2,1)
&=\int_0^t\int_0^{r}\exp_{+,\,s}(A_{33})A_{32}\exp_{+,\,r-s}(A_{22})A_{21}\exp_{+,\,t-r}(A_{11})\,ds\,dr.\notag
\end{align}
\end{example}

\begin{remark}\label{r:Saad}
Integral~\eqref{e:exp[1] from Bellman} was first obtained in~\cite[ch.~10, \S~14, formula (5)]{Bellman60:eng} in a different context. Formula~\eqref{e:exp(A) via exp[i]} for a triangular block matrix (with blocks consisting of scalars) of the size less than or equal to $4\times 4$ was appeared in~\cite[theorem 1]{Van_Loan78} and for a triangular block matrix of any size in~\cite{Carbonell-Jimenez-Pedroso08}.
\end{remark}

\begin{example}\label{ex:g_t of block triang matr}
Let $A$ be the block matrix~\eqref{e:matrix A},
whose spectrum is disjoint from the imaginary axis.
Then for $t\neq0$, we have
\begin{equation*}
g_t(A)=\begin{pmatrix}
g_t(A_{11}) & 0 & 0 \\
g_t^{[1]}(A;2,1) & g_t(A_{22}) & 0 \\
g_t^{[1]}(A;3,1)+g_t^{[2]}(A;3,2,1) &g_t^{[1]}(A;3,2) & g_t(A_{33}) \\
 \end{pmatrix},
\end{equation*}
where
\begin{align*}
g_t^{[1]}(A;i_1,i_2)
&=\int_{-\infty}^{\infty}g_s(A_{i_1,i_1})A_{i_1,i_2}g_{t-s}(A_{i_2,i_2})\,ds,\label{e:g_t[1]} \\
g_t^{[2]}(A;3,2,1)
&=\int_{-\infty}^{\infty}\int_{-\infty}^{\infty}
g_{s}(A_{33})A_{32}g_{r-s}(A_{22})A_{21}g_{t-r}(A_{11})\,ds\,dr.\notag
\end{align*}
\end{example}

\section*{Acknowledgments}
The first author was supported by the Ministry of Education and Science of the Russian Federation under state order No.~3.1761.2017/4.6. The second author was supported by the Russian Foundation for Basic Research under research projects No. 16-01-00197.

\providecommand{\bysame}{\leavevmode\hbox to3em{\hrulefill}\thinspace}
\providecommand{\MR}{\relax\ifhmode\unskip\space\fi MR }
\providecommand{\MRhref}[2]{%
  \href{http://www.ams.org/mathscinet-getitem?mr=#1}{#2}
}
\providecommand{\href}[2]{#2}

\vskip5mm
\end{document}